\documentclass[12pt]{article}

\usepackage[margin=1in]{geometry}  
\usepackage{graphicx}              
\usepackage{amsmath}               
\usepackage{amssymb}               
\usepackage{amsfonts}              
\usepackage{amsthm}                
\usepackage{mathtools}
\usepackage{enumerate}

\newtheorem{theorem}{Theorem}[section]
\newtheorem{lemma}[theorem]{Lemma}
\newtheorem{claim}[theorem]{Claim}

\newtheorem{corollary}[theorem]{Corollary}

\newtheorem{definition}[theorem]{Definition}
\newtheorem*{theorem*}{Theorem}

\DeclareMathOperator{\tw}{tw}

\DeclarePairedDelimiter\abs{\lvert}{\rvert}%
\usepackage{color}
\usepackage[usenames,dvipsnames]{xcolor}

\newcommand{\ep}{Erd\H{o}s-P\'osa}

\title{A tight Erd\H{o}s-P\'osa function for long cycles}
\author{F. Mousset$^\ddag$\thanks{author was supported by grant no.\ 6910960 of the Fonds National de la Recherche, Luxembourg.}, A. Noever$^\ddag$\footnote{author was supported by grant no.\ 200021 143338  of the Swiss National Science Foundation.}, N. \v{S}kori\'c$^\ddag$\!\!, and F. Weissenberger\thanks{Department of Computer Science, ETH
Z\"urich, 8092 Z\"urich, Switzerland. \newline Email: {\tt
\{frank.mousset|anoever|nskoric|felix.weissenberger\}@inf.ethz.ch}.}
}

\begin{document}

\maketitle

\abstract{
A classic result of Erd\H{o}s and P\'osa says that any graph contains either
$k$ vertex-disjoint cycles or can be made acyclic by deleting at most $O(k \log
k)$ vertices. Here we generalize this result by showing that
for all numbers $k$ and $l$ and for every graph $G$, either $G$ contains $k$
vertex-disjoint cycles of length at least $l$, or there exists a set $X$ of
$\mathcal O(kl+k\log k)$ vertices that meets all cycles of length at least $l$
in $G$. As a corollary, the tree-width of any graph $G$ that does not contain
$k$ vertex-disjoint cycles of length at least $l$ is of order $\mathcal
O(kl+k\log k)$. These results improve on the work of Birmel\'e, Bondy and Reed
'07 and Fiorini and Herinckx '14 and are optimal up to constant factors.}

\section{Introduction}\label{sec:intro}

Let $\mathcal F$ be any family of graphs. Given a graph $G$, a subset
$X\subseteq V(G)$ is called a \emph{transversal} (of $\mathcal F$) if the graph
$G-X$ obtained by deleting $X$ does not contain any member of $\mathcal F$. We
say that $\mathcal{F}$ has the \emph{Erd\H{o}s-P\'osa property} if there exists
a function $f\colon\mathbb{N}\to \mathbb R$ such that every graph $G$ which
does not contain $k$ vertex-disjoint members of $\mathcal{F}$ contains a
transversal of size at most $f(k)$.

The study of this property dates back to 1965 when Erd\H{o}s and P\'osa
\cite{erdos65} showed the following:

\begin{theorem*}\label{thm:erdos-posa}
  Every graph contains either $k$ vertex-disjoint cycles or a set of
  at most $f(k) = (4 + o(1)) k \log k$ vertices meeting all its cycles.
\end{theorem*}

The value of $f(k)$ in this theorem is optimal up to the constant factors. The
\ep{} property is closely related to classical `covering vs.\ packing'
results in graph theory, such as K\H{o}nig's theorem or Menger's theorem. For
example, K\H{o}nig's theorem can be stated as follows: every bipartite graph contains
either $k$ vertex-disjoint edges or a set of $f(k)=k$ vertices meeting all the
edges. The above result has spawned a long line of papers about the
duality between packing and covering of different families of graphs, directed
graphs, hypergraphs, rooted graphs, and other combinatorial objects (see a
recent survey of Raymond and Thilikos \cite{2016survey} for more information).


In this paper, we are interested in the Erd\H{o}s-P\'osa property for the
family $\mathcal F_l = \{C_m \mid m \geq l\}$ of cycles of length at least $l$.
A 1988 result of Thomassen~\cite{thomassen88} implies that for every $l$, the
family $\mathcal F_l$ has the Erd\H{o}s-P\'osa property with a function $f(l,
k)\in 2^{l^{\mathcal O(k)}}$ (though recent results of Chekuri and Chuzhoy make
it possible to substantially improve the dependency on $k$ in this
bound~\cite{chekuri13}). This result was sharpened by Birmel\'e, Bondy, and
Reed~\cite{bondy07} to $f(l, k) \in \mathcal{O}(lk^2)$ in 2007 and by Fiorini
and Herinckx~\cite{fiorini14} to $f(l, k) \in \mathcal{O}(lk\log k)$ in 2014.
In this paper we improve these results to the asymptotically optimal bound
$f(l, k)\in \mathcal O(kl+k\log k)$, thus settling the question asked
in~\cite{bondy07} and~\cite{fiorini14}.

\begin{theorem}\label{thm:main}
  For every integer $l \ge 3$, the family $\mathcal{F}_l$
  of cycles of length at least $l$ has the Erd\H{o}s-P\'osa
  property with the function
  \[ f(l, k) = \begin{cases}
    6kl + 10k\log_2 k + 40k + 10k\log_2\log_2 k &\text{if }k\geq 2,\\
    0 & \text{if }k=1.
  \end{cases} \]
\end{theorem}

Certainly the constant factors in Theorem~\ref{thm:main} are not
optimal. Birmel\'e, Bondy, and Reed~\cite{bondy07} conjectured the
correct function in the case $k=2$ to be $f(l, 2) = l$. The complete graph on
$2l-1$ vertices shows that, if true, this bound would be tight.
Lov\'asz~\cite{lovasz65} confirmed the conjecture for $l=3$, while
Birmel\'e~\cite{birmele03} confirmed the cases $l=4$ and $l=5$. For larger $l$,
Birmel\'e, Bondy, and Reed~\cite{bondy07} proved that the optimal function
satisfies $f(l, 2)\leq 2l+3$. This was recently improved by Meierling,
Rautenbach and Sasse~\cite{meierling14} to $f(l, 2)\leq 5l/3+29/2$.

There are two constructions which together imply that the function in
Theorem~\ref{thm:main} is asymptotically optimal for large $k$ and $l$. On the
one hand, for all $k$ and $l$ we must have $f(l, k) \geq (k-1)l$, as can be seen
from the example of a complete graph on $kl-1$ vertices: this graph does not contain $k$ vertex-disjoint
cycles of length at least $l$, but to remove all cycles of length at least $l$ one must
delete $kl-1 - (l-1)= (k-1)l$ vertices. This construction also gives the lower
bound $f(l, k) \geq \frac12(k-1) \log_2 k$ whenever $l \geq \frac12\log_2 k$.

On the other hand, for $l < \frac12\log_2 k$ we can obtain the lower bound
$f(l, k) \geq \frac18 k\log_2 k$ using the fact that there exist $3$-regular
graphs on $n$ vertices with girth at least $(1-o(1))\log_2 n$~\cite{erdos63}.
Indeed for $n$ large enough, let $G$ denote such a graph with girth $g(G)$.
Clearly $G$ contains at most $n/g(G)$ vertex disjoint cycles. So fix $k=\lfloor
n/g(G) \rfloor + 1>n/g(G)$ and observe that for $n$ large enough $n\geq \frac12 k
\log_2k$. All cycles in $G$ have length at least $g(n)>\frac{1}{2}\log_2 k>l$.
Thus if $X$ is a transversal of all cycles of length at least $l$ then $G-X$ is
a forest.  Because $G$ is $3$-regular, removing $|X|$ vertices leaves at least
$3n/2 -3|X|$ edges. Since the resulting graph should be a forest, we need
$\frac 32 n - 3|X| \le n - |X|$, and therefore every transversal must have size
$|X| \ge \frac n4\geq \frac18 k\log_2 k$. This gives the desired lower bound
$f(l,k) \geq \frac 18 k\log_2k$.

\paragraph{Notation}
All graphs are assumed to be simple unless stated otherwise. However,
multigraphs do make an appearance in the proof. We define a multigraph $M$ in
the standard way, that is, as an ordered pair $(V, E)$, where $V$ denotes the
vertex set of $M$ and $E$ is the multiset of edges of $M$. For a (multi-)graph
$G$ we denote by $V(G)$ and $E(G)$ the vertex set and the edge (multi-)set of $G$,
respectively. Given two multigraphs $M_1$ and $M_2$ we write $M_1\cup M_2$ for
the multigraph $M=(V,E)$ where $V = V(M_1) \cup V(M_2)$ and $E = E(M_1)\cup E(M_2)$.
In particular, the multiplicity of an edge $e$ in $M_1\cup M_2$ is equal to the
sum of multiplicities of $e$ in $M_1$ and $M_2$. We use the standard asymptotic
notation $\mathcal{O}$, $o$, $\omega$ and $\Omega$.

\paragraph{Tree-width}
Our results also imply an asymptotically optimal upper bound on the tree-width
$\tw(G)$ of every graph $G$ that does not contain $k$ vertex-disjoint cycles of
length at least $l$. We need the following theorem.

\begin{theorem}[Birmel\'e~\cite{birmele02}]\label{thm:birm}
  Suppose that $G$ does not contain a cycle of length at least $l$.
  Then $\tw(G) \leq l-2$.
\end{theorem}

Generalizing this, Birmel\'e, Bondy, and Reed proved
that any graph $G$ not containing $k$ vertex-disjoint cycles of length
at least $l$ has tree-width in $\mathcal O(k^2l)$~\cite{bondy07}.
Theorem~\ref{thm:main} allows us to improve this bound:

\begin{corollary}
  Assume that $G$ does not contain $k$ vertex-disjoint
  cycles of length at least $l$. Then
  $\tw(G) \in \mathcal O(kl + k\log k)$.
\end{corollary}
\begin{proof}
  Assume that $G$ does not contain $k$ vertex-disjoint cycles of length at least
  $l$. By Theorem~\ref{thm:main} there is a set $X\subseteq V(G)$ of size $|X|
  \leq 6kl + 10k\log_2 k+40k + 10\log_2\log_2 k$ such that $G-X$ does not
  contain a cycle of length at least $l$. By Theorem~\ref{thm:birm} we have
  $\tw(G-X)\leq l-2$. We can turn a tree-decomposition of $G-X$ into a
  tree-decomposition of $G$ by adding $X$ to each bag, which gives the bound
  $\tw(G) \leq \tw(G-X) + |X| \leq (6k+1)l+ 10k\log_2 k + 40k + 10\log_2\log_2
  k-2$.
\end{proof}

This is tight in the sense that there are examples of graphs that do not
contain $k$ disjoint cycles of length at least $l$ and whose tree-widths are in
$\Omega(kl+k\log k)$.

In fact, similar constructions as above work. An example where $\tw(G)\geq
kl-2$ is provided by the complete graph on $kl-1$ vertices. For $l \le c \log
k$, for sufficiently small positive constant $c$, we can use fact that there
exist constant-degree expander graphs $G$ on $n$ vertices with $g(G)\in
\Omega(\log n)$ and $\tw(G)\in \Omega(n)$ (using for example the results
in~\cite{lps88} and~\cite{bptw10}). Choosing $k$ such that $ k \cdot g(G) \in
[n+1, 2n]$ one obtains a graph which does not contain $k$ vertex-disjoint cycles
(of any length) but whose tree-width is in $\Omega(k\log k)$.

\section{Proof of the main result}\label{sec:proof}

We will use the following lemma.

\begin{lemma}[Diestel~\cite{diestel05}]\label{lem:cubic}
  For each natural number $k$, let
  \[ s_k := \begin{cases}4k(\log_2 k + \log_2\log_2 k + 4) &\text{if } k \ge 2, \\
    1& \text{if }k=1.
  \end{cases}
  \]
  Then every $3$-regular multigraph on at least $s_k$ vertices contains
  a set of $k$ vertex-disjoint cycles.
\end{lemma}

Fix $l\geq 3$ and a graph $G$. We say that a cycle in $G$ is \emph{long} if it has
length at least $l$, and otherwise we say that it is \emph{short}. By
\emph{disjoint}, we always mean \emph{vertex-disjoint}.
We assume that $G$ does not contain $k$ disjoint long cycles and
show that $G$ contains a transversal of $\mathcal F_l$
of size at most $f(l,k)$, where
\[ f(l,k) = \begin{cases}6kl + 10k\log_2 k + 40k + 10k\log_2\log_2 k &\text{if }k\geq 2,\\
  0 & \text{if }k=1.
\end{cases}\]
The proof is by induction on $k$, where the base case $k=1$ is obvious.

If $G$ contains a long cycle $C$ of length at most $6l$ then by induction,
$G-V(C)$ contains either $k-1$ disjoint long cycles or a transversal $X$ of
size $f(l,k-1)$. In the first case $G$ contains $k$ disjoint long cycles and in
the second case $X\cup V(C)$ is a transversal of size $f(l,k-1)+6l\leq f(l,k)$.
Therefore we may assume that every long cycle in $G$ contains strictly more
than $6l$ vertices.

Let $H$ denote a maximal subgraph of $G$ with the following properties:
\begin{enumerate}
  \item all vertices of $H$ have degree $2$ or $3$ in $H$;
  \item $H$ contains no short cycle.
\end{enumerate}
Similarly as in \cite{erdos65}, observe that $H$ is the union of a subdivision
of a $3$-regular multigraph and at most $k-1$ disjoint long cycles. If $H$
contains at least $s_k$ vertices of degree $3$ then by Lemma~\ref{lem:cubic},
it contains $k$ disjoint cycles, which by definition of $H$ are all long. So
from now on, we can assume that $H$ contains fewer than $s_k$ vertices of
degree $3$.

\begin{definition}
  We say that a path $P$ is an \emph{$H$-path} if its endpoints are distinct vertices of $H$ and
  if it is internally vertex-disjoint from $H$. Observe that we allow for
  $P\subseteq H$ if the length of $P$ is one. We say that $P$ is a \emph{proper
  $H$-path} if none of its edges are contained in $H$.
\end{definition}

For each $i\in\{2,3\}$ let $V_i\subseteq V(H)$ denote the set of vertices with
degree $i$ in $H$. We modify $G$ by removing all edges from $E(G)\setminus
E(H)$ that are incident to a vertex from $V_3$. Note that any transversal of
the modified graph can be turned into a transversal of the original graph by
additionally removing $V_3$. Furthermore $H$ is still maximal in the modified graph.  From
now on we assume that all vertices of $V_3$ have degree $3$ in $G$. In
particular the endpoints of every proper $H$-path lie in $V_2$.

This implies that every $H$-path has length at most $l$, as otherwise we could
add the path to $H$ without violating either the degree or the cycle condition,
contradicting the maximality. For the same reason, if $P$ is an $H$-path with
endpoints $s,t\in V(H)$, then there exists a path between $s$ and $t$ in $H$ of length
at most $l$. In fact, as $H$ contains no cycles of length at most $2l$, this
path is unique.
Thus the following notion is well-defined.

\begin{definition}[Projection]
  Suppose that $P$ is an $H$-path with endpoints $s, t\in V(H)$. The
  \emph{projection} of $P$, denoted by $\pi(P)$, is defined to be the unique
  path of length at most $l$ between $s$ and $t$ in $H$.

  Let $C \subseteq G$ be a cycle in $G$ that intersects $H$ in at least two
  vertices. We define the projection $\pi(C)$ of $C$ as follows. Let $C =
  P_1\cup \dotsc \cup P_m$ be a decomposition of $C$ into distinct $H$-paths.
  Then we define the projection of $C$ to be the multigraph $\pi(C) =
  \pi(P_1)\cup \dotsb \cup \pi(P_m)$.

  If $P$ is a path in $G$ with distinct endpoints in $H$ (not necessarily an
  $H$-path), then we define the projection analogously: let $P= P_1\cup \dotsb \cup P_m$
  be a decomposition into  $H$-paths and define $\pi(P) = \pi(P_1)\cup \dotsb
  \cup \pi(P_m)$.
\end{definition}

We remark that in the definition above, the decomposition of a cycle or path
into distinct $H$-paths is unique up to permutation, so that the projection is
in fact well-defined.

We claim that information about the length of a cycle can be recovered by
looking at the following property of its projection:
\begin{definition}
  A multigraph $M$ is called \emph{even} if the multiplicity of every edge in
  $M$ is even.
\end{definition}

\begin{lemma}\label{lemma:projection}
  Suppose that $C$ is a cycle in  $G$ which intersects $H$ in at least two vertices. If
  $V(\pi(C))$ induces a tree in $H$, then $C$ is short. If $\pi(C)$ is not
  even, then $C$ is long.

\end{lemma}

\begin{proof}
  Among all cycles for which the lemma fails, we may pick a cycle $C$ whose
  decomposition into $H$-paths minimizes the number of proper $H$-paths. If
  none of these $H$-paths is proper, then $C\subseteq H$ and thus $\pi(C)=C$ is not even and long.
  Therefore we may assume that $C$ contains at least one proper
  $H$-path $P$.

  Note that $\pi(P)$ is a path in $H$ with the same endpoints as $P$. If $C$
  contains all edges of $\pi(P)$, then we actually have $C=P\cup \pi(P)$ and so
  the projection of $C$ is even. Moreover,
  the length of  $C$ is $|P|+|\pi(P)|\leq 2l$. Since all long cycles have
  length greater than $6l$ we see that $C$ must be short and we are done.
  Therefore, we can assume that at least one edge of $\pi(P)$ does not belong
  to $C$. Since $\pi(P)$ is a path with endpoints in $V(C)$, there exists a
  path $P'\subseteq \pi(P)$ with endpoints $s,t\in C$ which is internally
  vertex-disjoint from $C$ and whose edges are not edges of $C$. Let $P_1, P_2
  \subseteq C$ denote the two internally disjoint $s,t$-paths in $C$ and
  consider the two cycles $C_1\coloneqq P_1\cup P'$ and $C_2\coloneqq P_2\cup
  P'$. Observe that since $\pi(P')\subseteq \pi(C)$ we have
  $V(\pi(C))=V(\pi(C_1))\cup V(\pi(C_2))$. Additionally, the parity of each
  edge in $\pi(C)$ is equal to the parity of the same edge in $\pi(C_1)\cup
  \pi(C_2)$.

  We are now ready to prove the first claim. Assume that $C$ is long and that
  $H[V(\pi(C))]$ is a tree. Then the projections of $C_1$ and $C_2$ induce
  trees as well and therefore both cycles contain at least one proper $H$-path.
  In particular both cycles contain strictly fewer proper $H$-paths than $C$.
  Furthermore, at least one of the two cycles has length at least $\abs C/2$.
  Since $C$ has length at least $6 l$, this cycle is still long, which
  contradicts the minimality of $C$.

  For the second claim, assume that $C$ is short and that $\pi(C)$ is
  not even. Observe that we have  $|C_i| \le |C| + |P'| \le l + |P'|$ for
  each $i\in\{1,2\}$. As $P'$ is a subgraph of $\pi(P)$ we know that $|C_i| \le
  2 l$, which implies that both cycles $C_i$ are short. Since $H$ contains only
  long cycles, this means that both $C_1$ and $C_2$ contain at least one proper
  $H$-path, so both $C_1$ and $C_2$ contain fewer proper $H$-paths than $C$.
  Finally, as $\pi(C)$ is not even and since each edge in $\pi(C_1)\cup \pi(C_2)$
  has the same parity as the same edge in $\pi(C)$, at least one of
  $\pi(C_1)$ and $\pi(C_2)$ is not even. This contradicts the minimality
  of $C$.
\end{proof}



\begin{definition}
  Let $X \subseteq V(H) \cup E(H)$ be a set of vertices and edges of $H$. Let
  us denote by $G - X$ and $H-X$ the graphs obtained by deleting from $G$ and
  $H$ the edges and vertices in $X$.

  Then we say that $X$ is \emph{$\pi$-preserving} if $P \subseteq G-X$ implies
  $\pi(P)\subseteq H-X$ for all $H$-paths $P$.
\end{definition}

The following lemma is the crucial
ingredient of the proof of Theorem \ref{thm:main}. For now, we only state the
lemma; the proof is given in the next subsection.

\begin{lemma}\label{lemma:forest-cut}
  There exists a $\pi$-preserving set $X$ for which $H-X$ is a forest and such
  that $|X \cap (V_2\cup E(H[V_2]))| \leq 3|V_3|/2+k$.
\end{lemma}

We can now finish the proof of Theorem~\ref{thm:main}. By the lemma, there exists a
$\pi$-preserving set  $X$ for which $H-X$ is a forest and such that
$|X\cap (V_2\cup E(H[V_2]))|\leq 3|V_3|/2+k$. Suppose that $C$ is a cycle that intersects
$H$ at least twice. Then because $X$ is $\pi$-preserving, we know in
particular that all vertices of $V(\pi(C))$ are contained in the same component
of $H-X$. Since each component of $H-X$ is a tree, the graph $H[V(\pi(C))]$
must also be a tree. Thus, by Lemma~\ref{lemma:projection}, the cycle $C$ is
short. It follows that every long cycle in $G-X$ intersects $H$ at most once.

To construct the transversal we define the following two sets.
\begin{enumerate}
  \item Let $X'\subseteq V(H)$ be a set containing the vertices in $X\cap V_2$ and also
    containing one endpoint of each edge in $X\cap E(H[V_2])$.
  \item Let $Z\subseteq V(H)$ denote the set of all vertices $z\in V(H)$ for
    which there exists some long cycle $C_z$ such that $V(C_z)\cap
    V(H)=\left\{z\right\}$.
\end{enumerate}

We now claim that $V_3\cup X'\cup Z$ is a transversal of all long cycles. Every long
cycle $C$ intersects $H$ at least once since otherwise $C$ could be added to
$H$. If $C$ intersects $H$ exactly once than it intersects $Z$. If $C$
intersects $H$ at least twice then, by the observation above, this means that
$C$ intersects $X$. But then $C$ must intersect either $V_3$ or $X'$. So
$V_3\cup X'\cup Z$ is a transversal in $G$. Recall that in the beginning, we
modified the graph $G$ by removing all edges of $E(G)\setminus E(H)$ that are
incident to a vertex of $V_3$. Since we remove $V_3$ anyway, $V_3\cup X'\cup Z$
is also a transversal in the original graph.

It remains to bound the size of this transversal. If for some $z\neq z'\in Z$
the cycles $C_z$, $C_{z'}$ intersect, then one can see that the assumption
$|C_z|,|C_{z'}| \geq 6l$ implies that $C_z\cup C_{z'}$ contains a $z$-$z'$-path
of length at least $l$, which contradicts the fact that there are no $H$-paths
of length $l$ or longer. Therefore $\left\{C_z\mid z\in Z\right\}$ is a
collection of vertex disjoint long cycles and in particular $\abs{Z} < k$.
Furthermore, we have $|X'| \leq 3|V_3|/2+k$. Since $|V_3|< s_k$, we get \[
|V_3\cup X'\cup Z| < |V_3| + \frac32|V_3| + 2k < \frac52s_k + 2k \leq f(l,k).
\] This completes the proof of Theorem~\ref{thm:main}. However, we still need
to prove Lemma~\ref{lemma:forest-cut}.

\subsection{Proof of Lemma \ref{lemma:forest-cut} }

Let us call a $\pi$-preserving set \emph{valid} if it contains at most one
edge or vertex from every component of $H[V_2]$.
There exists at least one valid $\pi$-preserving set: the empty set. Let $X$
denote a $\pi$-preserving set of maximal size. Since $H$ is the disjoint union
of a subdivision of a $3$-regular graph on $|V_3|$ vertices and fewer than
$k$ cycles, the fact that $X$ is valid immediately implies that $|X\cap (V_2\cup E(H[V_2]))| \leq
3|V_3|/2 + k$. We will show that $H-X$ is a forest. Assume
towards a contradiction that $H-X$ contains a cycle. Among all such cycles let
$C=(x_0\dots,x_n)$ denote one of minimum length. Since $H$ only contains long cycles, we
have $|C|>6l$. Let $B_{H-X}(C,l)$ denote the ball of radius $l$ around $C$ in
$H-X$, i.e., the set of vertices that have distance at most $l$ to a vertex of
$C$ in $H-X$. Let $H_C$ be the subgraph of $H-X$ induced by $B_{H-X}(C,l)$. We
now have some claims about the structure of $H_C$. First of all, it is clear
that $C\subseteq H_C$. Moreover:

\begin{claim}\label{claim:1}
  In $H_C$, every vertex $v$ has a unique nearest vertex in $C$ (which may be
  $v$ itself if $v\in V(C)$).
\end{claim}
\begin{proof}
  This is clear if $v\in V(C)$, so assume otherwise. By the definition of
  $H_C$, the distance from $v$ to any nearest vertex on $C$ is at most $l$.
  If there are two nearest vertices of $v$ on $C$, then using the shortest
  path between them on $C$, we obtain a cycle
  of length at most $|C|/2 +2l < |C|$ in $H_C$, where the inequality follows from $|C|>4l$. But
  this would contradict the minimimality of $C$.
\end{proof}

This claim shows in particular that the projection map $p\colon V(H_C)\to V(C)$
taking vertices of $H_C$ to their unique nearest vertex in $C$ is well-defined.
In fact, the preimages of this projection have rather nice properties:

\begin{claim}\label{claim:structure}
  The following hold:
  \begin{enumerate}[(i)]
    \item for every $x\in V(C)$, the graph $H_C[p^{-1}(x)]$ is a tree;
    \item for distinct vertices $x,y\in C$, there are no edges between $p^{-1}(x)$
      and $p^{-1}(y)$ in $H_C\setminus C$;
  \end{enumerate}
\end{claim}
\begin{proof}
  For (i) simply observe that the diameter of $H_C[p^{-1}(x)]$ is at most $2l$, so
  any cycle would be of length at most $4l$; however $H$ does not contain cycles that are this short.
  The argument for (ii) is similar to the argument in the proof of Claim~\ref{claim:1}: if such
  an edge exists, then $H_C$ contains a cycle of length at most $|C|/2+2l+1 < |C|$, using
  $|C|>6l > 4l+2$. This would contradict the minimality of $C$.
\end{proof}

By the above claim, the graph $H_C$ is just the cycle $C$ with
trees attached at every vertex. In particular, $H_C-e$ is a tree for any edge $e\in E(C)$.

We will now construct a larger graph $G_C$ where $H_C\subseteq G_C\subseteq G-X$ as follows.
Let $D = \{x_l,x_{l+1},\dotsc,x_{n-l}\}$.
For each $x\in D$, let $P(x)$ be the set of all $H$-paths $P$ in $G-X$ such
that $x \in V(\pi(P))$.  Since $X$ is $\pi$-preserving, every such path satisfies
$\pi(P)\subseteq H-X$. We then define
\[ G_C = H_C \cup \bigcup_{x\in D} P(x). \]
It is worth noting that $G_C \cap H = H_C$.
Let us further denote by $e^*$ the edge $\{x_0, x_n\}$.
We have the following very important claim, whose proof we postpone to the end of the section.

\begin{claim}\label{claim:obvious}
For every proper $H$-path $P\subseteq G_C$, we have $\pi(P)\subseteq H_C-e^*$.
  In particular, every path in $G_C-e^*$ with endpoints in $V(H_C)$ projects
  to a subgraph of $H_C-e^*$.
\end{claim}

Using the claim, we can now finish the proof of Lemma~\ref{lemma:forest-cut}.
Let $A=\left\{x_0,\dots,x_l\right\}$ and
$B=\left\{x_{n-l},\dots,x_n\right\}$. We claim that $G_C-e^*$ does not contain
two internally disjoint $A$-$B$-paths. Suppose for a contradiction that $P_1$
and $P_2$ are two such paths, where we can assume that both $P_1$ and $P_2$
intersect both $A$ and $B$ in exactly one vertex. By Claim~\ref{claim:obvious}
both paths project onto a subgraph of $H_C-e^*$, which implies in particular
that $\pi(P_1)\cup \pi(P_2)$ does not contain $e^*$. We can combine $P_1$ and
$P_2$ into a cycle in $G_C-e^*$ by adding the shortest paths between the
endpoints of $P_1$ and $P_2$ in $C[A]$ and $C[B]$, respectively. The projection
of this cycle lies in $H_C-e^*$ which is a tree. Thus, by
Lemma~\ref{lemma:projection}, this cycle is short. In particular we have $|P_1|
< l$. However, as the projection of $P_1$ avoids $e^*$, connecting the
endpoints of $P_1$ using the shortest path in $C[A\cup B]$ results in a cycle
$C_1$ whose projection is not even, since the multiplicity of $e^*$ in
$\pi(C_1)$ is one. So by Lemma \ref{lemma:projection} the cycle $C_1$ is long.
Since any path in $C[A\cup B]$ is of length at most $2l+1$,  the length of
$C_1$ is bounded by $\abs {P_1} + 2l+1 \leq 3l$. This is a contradiction since
$G$ does not contain a long cycle of length at most $3l$. We conclude that
there are no two internally disjoint $A$-$B$-paths in $G_C-e^*$.

By Menger's theorem, $G_C-e^*$ contains a single-vertex $A$-$B$-cut. Denote the
cut vertex by $x$. Because $C$ contains an $A$-$B$-path from $x_l$ to
$x_{n-l}$, we must have $x\in D$. Thus $x$ has degree at least two in $H_C-e^*$.
We distinguish two cases, depending on whether $x$ has degree two or three in
$H_C-e^*$.

First, assume that $x$ is a vertex of degree two in $H_C-e^*$. By the
maximality of $X$, we know that $X\cup \{x\}$ is not a valid $\pi$-preserving
set. One possibility is that $X\cup \{x\}$ is $\pi$-preserving, but it
is not valid. Since $X$ is valid on its own, this means that $x$ belongs to a
component of $H[V_2]$ which intersects $X$. But since $x$ belongs to the cycle
$C\subseteq H-X$, this is impossible. Thus it must be that $X\cup \{x\}$ is not
a $\pi$-preserving set. Since $X$ by itself is $\pi$-preserving, the definition
implies that there exists an $H$-path $P$ in $G-X$ such that $x \in V(\pi(P))$
and $x \notin V(P)$. Note that $P\subseteq G_C-e^*$ by the definition of $G_C$.
Since $x$ has degree two in $H_C-e^*$ and since $H_C-e^*$ is a tree, we know
that $H_C-e^* - x$ breaks into exactly two trees $T_1$ and $T_2$. As $x \in
V(\pi(P))$ and $\pi(P)\subseteq H_C-e^*$, it must be that $P$ has one endpoint
in $T_1$ and the other in $T_2$. It follows that $T_1 \cup T_2 \cup P$ is a
connected graph and thus $T_1 \cup T_2 \cup P$ must contain an $A$-$B$-path.
This is a contradiction with the fact that $x$ separates $A$ and $B$ in
$G_C-e^*$ and completes the proof in this case.

Next, assume that $x$ is a vertex of degree three in $H_C-e^*$. Recall
that then $x$ has degree $3$ in $G$ as well, and in particular no proper
$H$-path has $x$ as an endpoint. As in the previous case, $H_C -e^*- x$ breaks
into exactly three trees $T_1, T_2$ and $T_3$. Let $x^-$ and $x^+$ be the
neighbours of $x$ on $C$. Without loss of generality assume that $x^-\in T_1$
and $x^+\in T_3$. Consider any $H$-path $P\subseteq G-X$ such that $x\in
V(\pi(P))$ (so in particular $P\subseteq G_C-e^*$). Since $\pi(P)$ is a
path contained in $H_C-e^*$, the endpoints of $P$ must be in two
different trees. However, it cannot happen that one endpoint of $P$ is in $T_1$
and the other in $T_3$, since then $T_1 \cup T_3 \cup P$ would contain an
$A$-$B$-path, contradicting the fact that $x$ is an $A$-$B$-cut vertex in
$G_C-e^*$. For the same reason there are no two $H$-paths $P_1,
P_2\subseteq G-X$ such that $x\in V(\pi(P_1))\cap V(\pi(P_2))$ and such that
$P_1$ has endpoints in $T_1$ and $T_2$ and $P_2$ has endpoints in $T_2$ and
$T_3$. We conclude that there is some $j\in \{1,3\}$ such that all $H$-paths
$P$ with $x\in V(\pi(P))$ have one endpoint in $T_2$ and the other endpoint in
$T_j$. Without loss of generality, assume $j=1$.
We now claim that by adding the edge $\{x, x^+\}$ to $X$ we get again a $\pi$-preserving set. If we assume otherwise, then there exists a $H$-path
$P$ in $G - (X\cup \{\{x,x^+\}\})$ whose projection contains $\{x, x^+\}$. This
is not possible, as such a path satisfies $x\in V(\pi(P))$ and has one endpoint
lying in $T_3$. Thus $X \cup \{\{x,x^+\}\}$ is $\pi$-preserving. In fact, since the edge
$\{x,x^+\}$ is incident to the degree-three vertex $x$, it is automatically a
valid $\pi$-preserving set. This contradicts the maximality of $X$ and completes the proof
of Lemma \ref{lemma:forest-cut}.

We now present the missing proof of Claim~\ref{claim:obvious}.

\begin{proof}[Proof of Claim \ref{claim:obvious}]
  The second statement clearly follows from the first.

  We start the proof with an observation about the $H$-paths in $G-X$ whose
  endpoints lie in $V(H_C)$. Let $P\subseteq G-X$ be such a path with endpoints
  $a,b\in V(H_C)$. Since $X$ is $\pi$-preserving, we have $\pi(P)\subseteq H-X$.
  Note that we know of at least one path in $H-X$ from $a$ to $b$: the path $Q
  = Q_1Q_2Q_3$, where $Q_1$ is the shortest path in $H_C$ from $a$ to $p(a)$,
  $Q_2$ is the shortest path from $p(a)$ to $p(b)$ on $C$, and $Q_3$ is the
  shortest path from $p(b)$ to $b$ in $H_C$. This path $Q$ is contained in
  $H_C$ and it has length at most $2l+|C|/2$.
  The important observation is that $\pi(P) = Q$. Indeed, if this were not so,
  then $Q\cup \pi(P) \subseteq H-X$ would contain a cycle of length at most
  $2l+|C|/2 + |\pi(P)| \leq 3l+|C|/2<|C|$ (using $|C|>6l$), contradicting the
  minimality of $C$.

  This shows in particular that for every $H$-path $P\subseteq G-X$ with
  endpoints in $V(H_C)$, we have $\pi(P)\subseteq H_C$. It remains to show the
  stronger statement that if additionally $P\subseteq G_C$, then
  $\pi(P)\subseteq H_C-e^*$.

  To obtain a contradiction, assume that $P\subseteq G_C$ is an $H$-path with
  endpoints $a,b\in V(H_C)$ such that $\{x_0,x_n\}\in E(\pi(P))$. Since the
  projection of $P$ has length at most $l$, and by the observation above, we
  can assume without loss of generality that $p(a)\in \{x_0,\dotsc,x_l\}$ and
  $p(b)\in \{x_{n-l},\dotsc,x_n\}$. Now let $c$ be the neighbour of $a$ on $P$.
  The edge $\{a,c\}$ belongs to $G_C$ but not to $H_C$, so there is some $x\in
  D=\{x_l,\dotsc,x_{n-l}\}$ and some $H$-path $P'\in P(x)$ such that
  $\{a,c\}\in E(P')$. Let $d$ be the other endpoint of $P'$ (one endpoint is
  $a$). Since $|C|>6l$, $x\in D$, $x\in V(\pi(P'))$, and $p(a)\in
  \{x_0,\dotsc,x_l\}$, we must have $p(d)\in \{x_l,\dotsc,x_{2l}\}$. The union
  of $P$ and $P'$ contains an $H$-path with endpoints $d$ and $b$. However, the
  projection of this $H$-path has length at least $\min{\{l+1,|C|-3l\}} > l$,
  which is impossible. \end{proof}

\section{Conclusions}

The main contribution of the paper is showing the asymptotically optimal
Erd\H{o}s-P\'osa function for the case of long cycles and growing $k$ and
$\ell$ and thus answering the question asked in \cite{bondy07} and
\cite{fiorini14}. We conclude the paper by mentioning some open problems:

\begin{itemize}
  \item{\textbf{$S$-cycles:}} Kakimura, Kawarabayashi, and
    Marx~\cite{kakimura2011} introduced a different generalization of the
    standard \ep{} theorem. They considered the family of $S$-cycles, i.e., all
    cycles of a graph which intersect a specified set $S$, and proved that such
    a family of cycles has the \ep{} property. Their result was later improved
    by Pontecorvi and Wollan~\cite{pontecorvi2012}, resulting in the following theorem:

    \begin{theorem}
      For any graph and any vertex subset $S$, the graph either contains $k$
      vertex-disjoint $S$-cycles or a vertex set of size $\mathcal{O}(k \log k)$
      that meets all $S$-cycles.
    \end{theorem}

    Since the vertex set $S$ can be the vertex set of the whole graph, this
    result is asymptotically tight. In 2014, Bruhn, Joos, and Schaudt
    \cite{bruhn2014long} combined the family of $S$-cycles with the family of
    long cycles and proved that the family of $S$-cycles of length at least
    $\ell$ has the \ep{} property with $f(k, \ell) = O(\ell k \log k)$. Thus,
    it is natural to ask if Theorem \ref{thm:main} generalizes to $S$-cycles as
    well.

  \item{\textbf{Edge-version:}} A family of graphs $\mathcal{F}$ is said to have the
    \emph{edge-\ep{}} property if there exists a function $f \colon \mathbb{N} \to
    \mathbb{R}$ such that every graph $G$ which does not contain $k$
    edge-disjoint members of $\mathcal{F}$ contains a set of $f(k)$ edges
    which meets all copies of members of $\mathcal{F}$ in $G$. Since the
    pioneering papers by  Erd\H{o}s and P\'osa \cite{erdos62,erdos65} it has
    been known that the family of cycles has the edge-\ep{} property as well.
    Namely, the following is true:
    \begin{theorem*}
      Any graph $G$ contains either $k$ edge-disjoint cycles or a set of $(2 +
      o(1)) k \log k$ edges meeting all its cycles.
    \end{theorem*}

    Pontecorvi and Wollan \cite{pontecorvi2012} generalized this result to the
    case of $S$-cycles by a clever reduction to the standard vertex-version of the
    problem. Already Birmel\'e et al.~\cite{bondy07} asked if the family of long
    cycles has the edge-\ep{} property. Unfortunately, the gadget trick from
    \cite{pontecorvi2012} breaks down in the case of long cycles and thus
    nothing is known in this scenario. It would be interesting to see if
    our approach could be applied for proving that the family of long cycles
    has the edge-\ep{} property.
\end{itemize}

\bibliographystyle{abbrv}
\bibliography{paper}

\end{document}